\documentclass[12pt]{article}

\usepackage{amssymb}
\usepackage{amsmath}
\usepackage{mathdots}
\usepackage{amsbsy}
\usepackage{amscd}
\usepackage{amsthm}
\usepackage{mathrsfs}
\usepackage{verbatim}
\usepackage[colorlinks]{hyperref}
\usepackage{fullpage}
\usepackage[english]{babel}
\usepackage{url}

\usepackage[english]{babel}
\usepackage[T1]{fontenc}
\usepackage[utf8]{inputenc}
\usepackage{tikz}
\usepackage{tikz-cd}
\usetikzlibrary{matrix}

\def\blue#1{{\bf \color{blue} #1}}

\theoremstyle{definition}
\newtheorem{Theorem}{Theorem}[section]
\newtheorem{Lemma}[Theorem]{Lemma}

\newtheorem{Remark}[Theorem]{Remark}
\newtheorem{Definition}[Theorem]{Definition}
\newtheorem{Corollary}[Theorem]{Corollary}
\newtheorem{Proposition}[Theorem]{Proposition}

\newcommand{\C}{\mathbb{C}}

\newcommand{\R}{\mathbb{R}}

\newcommand{\Z}{\mathbb{Z}}

\newcommand{\UU}{\mathbb{U}}
\newcommand{\PP}{\mathbb{P}}
\newcommand{\OO}{\mathbb{O}}

\newcommand{\DD}{\mathbb{D}}
\newcommand{\II}{\mathbb{I}}

\newcommand{\wil}[1]{\widetilde{#1}}

\newcommand{\mc}[1]{\mathcal{#1}} 
\newcommand{\mf}[1]{\mathfrak{#1}} 
\newcommand{\mt}[1]{\text{#1}}

\newcommand{\ii}{\textbf{i}}
\newcommand{\jj}{\textbf{j}}
\newcommand{\kk}{\textbf{k}}

\begin{document}

\title{Complex $\mt{G}_2$ and Associative Grassmannian}

\author{Selman Akbulut and Mahir Bilen Can}


\maketitle

\begin{abstract}
We obtain defining equations of the smooth equivariant compactification of the Grassmannian of 
the complex associative $3$-planes in $\C^7$, which is the 
parametrizing variety of all quaternionic subalgebras of the 
algebra of complex octonions $\OO\cong \C^8$. 
By studying the torus fixed points, we compute the Poincar\'e polynomial 
of the compactification. 

\vspace{.2cm}
\noindent
\textit{Keywords:} Associative grassmannian, octonions, quaternions.
\noindent 

\end{abstract}

\section{Introduction}

The exceptional Lie group $\mt{G}_2$, 
similar to any other Lie group, has different guises 
depending on the underlying field; 
it has two real forms and a complex form.
All of these incarnations have descriptions 
as a stabilizer group. We denote these three forms by $\mt{G}_{2}^+$,
$\mt{G}_{2}^-$, and by $\mt{G}_2$, respectively, 
where first two are real forms. 
The first of these is compact, connected, simple, 
simply connected, of (real) dimension 14, 
the second group is non-compact, connected, 
simple, of (real) dimension 14. 
The complex form of $\mt{G}_2$ is non-compact, 
connected, simple, simply connected, of (complex) dimension 14. 

\vspace{.05in}

Let $V$ denote either $\R^7$ or $\C^7$, and $e_1,\dots, e_7$ be its standard basis, and  
$x_1=e_1^*,\dots,x_7=e_7^*$ denote the dual basis. 
We write $\mt{GL}_7(\R)$ or $\mt{GL}_7(\C)$ instead of $\mt{GL}(V)$ when there is no danger of confusion.
If the underlying field does not play a role, then we write simply $\mt{GL}_7$.
Consider the fourth fundamental representation $\bigwedge^4 V$ of $\mt{GL}_7$, which is irreducible. 
Note that $\bigwedge^3 V^*$ is naturally isomorphic, as a $\mt{GL}_7$ 
representation, to $\bigwedge^4 V$, where $V^*$ denotes the dual space. 
On the other hand, $\bigwedge^4 V$ and $\bigwedge^3 V$ are dual representations, hence 
$\bigwedge^4 V \simeq (\bigwedge^3 V)^* \simeq \bigwedge^3 V^*$. 
Assuming that $i,j$ and $k$ are distinct numbers from $\{1,\dots,7\}$ let us use the shorthand $e^{ijk}$ 
to denote the wedge product $x_i \wedge x_j \wedge x_k$.
Following~\cite{Bryant87}, we set: 
$$
\phi^+:= e^{123} + e^{145} + e^{167}+e^{246} -e^{257} - e^{347} - e^{356},
$$
$$
{\phi}^-:= e^{123} - e^{145} - e^{167}-e^{246}  + e^{257} + e^{347} + e^{356}.
$$
It turns out that $\mt{G}_{2}^-$ is the stabilizer of $\phi^-$ in $\mt{GL}_7(\R)$, 
$\mt{G}_{2}^+$ is the stabilizer of $\phi^+$ in $\mt{GL}_7(\R)$, and finally, $\mt{G}_{2}$ is the stabilizer of $\phi:=\phi^+$ in $\mt{GL}_7(\C)$.
In fact, the orbit $\mt{GL}_7\cdot \phi$ is Zariski open in $\bigwedge^3 V^*$; over real numbers this orbit splits into two with stabilizers 
$\mt{G}_{2}^+$ and $\mt{G}_{2}^-$. A pleasant consequence of this openness is that $\mt{GL}_7(\C)$ has only finitely many orbits in $\bigwedge^3 V^*$. 

\vspace{.1in}

Let $\OO_k$ denote the octonion algebra over a field $k$. 
To ease our notation, when $k=\C$ we write $\OO$.
Following~\cite{SV}, we view $\OO_k$ as an 8-dimensional 
composition algebra; it is non-associative, unital (with unity $e\in \OO_k$), 
and it is endowed with a norm $N: \OO_k \rightarrow k$ such that 
$N(xy) = N(x) N(y) \ \text{ for all } x,y \in \OO_k$.

Composition algebras exist only in dimensions 
1,2,4 and 8. Moreover, they are uniquely determined 
(up to isotopy) by their quadratic form. 
When $k=\C$, there is a unique isomorphism 
class of quadratic forms and any member of this class is {\em isotropic}, 
that is to say the norm of the composition algebra vanishes on a nonzero element.
When $k=\R$ there are essentially two isomorphism classes of quadratic forms, 
first of which gives isotropic composition algebras, 
and the second class gives composition algebras with positive-definite quadratic forms.

\vspace{.1in}

For any real octonion algebra $\OO_\R$,the tensor product 
$\C \otimes_{\R} \OO_\R$ is isomorphic to $\OO$ 
(by the uniqueness of the octonion algebra over $\C$).
In literature $\C \otimes_{\R} \OO_\R=\OO$ is known as the 
``complex bioctonion algebra''.
We denote by $G_\R$ the group of algebra automorphisms 
of $\OO_\R$, and denote by $G$ the 
group of algebra automorphisms of $\OO$. 
By Proposition 2.4.6 of~\cite{SV} we know that the group of 
$\R$-rational points of $G$ is equal to $G_\R$. 
Of course(!), $G_\R$ is either $\mt{G}_2^+$ or $\mt{G}_2^-$ 
depending on which $\OO_k$ we start with. 
Meanwhile, $G$ is equal to $\mt{G}_2$. See Theorem 2.3.3,~\cite{SV}.

\vspace{.1in}

Let $N_1$ denote the restriction of the norm $N$ to $e_0^\perp$, the (7 dimensional) orthogonal 
complement of the identity vector $e_0\in \OO_k$. For notational ease we are going to denote $e_0^\perp$ in $\OO_k$
by $\II_k$ and denote $e_0^\perp$ in $\OO$ simply by $\II$.
The automorphism groups $G_\R$ and $G$ preserve the norms on their respective octonion algebras, and obviously 
any automorphism maps identity to identity. Thus, we know that $G_\R$ is contained 
in $SO(\II_k)$ and $G$ is contained in $SO(\II)$. Here, $SO(W)$ denotes the group of orthogonal transformations
of determinant 1 on a vector space $W$. If there is no danger of confusion, we write $SO_n$ ($n=\dim W$) in place of $SO(W)$.

\vspace{.1in}

A {\em quaternion algebra}, $\DD_k$ over a field $k$ is a 
4 dimensional composition algebra over $k$. 
As it is mentioned earlier, there are essentially (up to isomorphism) 
two quaternion algebras over $k=\R$ and there is a unique 
quaternion algebra over $k=\C$, which we denote simply by $\DD$, 
and call it the split quaternion algebra. (More generally, 
any composition algebra over $\C$ is called split.)
Any quaternion algebra over $\C$ is isomorphic 
to the algebra of $2\times 2$ matrices over $\C$ 
with determinant as its norm. 

\vspace{.1in}

The split octonion algebra $\OO$ has a description which is built on $\DD$ by the {\em Cayley-Dickson doubling process}:
As a vector space, $\OO$ is equal to $\DD\oplus \DD$ and its multiplicative structure is  
\begin{align}\label{A:octonion multiplication}
(a , b)(c , d) &= (ac + \bar d b,   da  + b\bar c),\ \text{ where } a,b,c,d\in \DD, 
\end{align}
and its norm is defined by $N((a,b))= N(a) - N(b)=\det a - \det b$.

\vspace{.1in}

Let $\mt{Gr}_k(3,\II_k)$ denote the grassmannian 
of 3 dimensional subspaces in $\II_k$.
Let $\DD_k\subset \OO_k$ be the quaternion 
subalgebra generated by the first four generators 
$e_1=e,e_2,e_3,e_4$ of $\OO_k$.
As usual, if $k=\C$, then we set $\DD = \DD_k$. 
Let us denote by $W_0$ the intersection $\DD_k \cap \II_k$, 
the span of $e_2,e_3,e_4$ in $\II_k$. 
By Corollary 2.2.4~\cite{SV}, when $k=\C$, 
we know that $\mt{G}_2$ acts transitively on the 
set of all quaternion subalgebras of $\OO$. 
Since an algebra automorphism fixes the identity, 
under this action, $W_0$ is mapped to another 
3-plane of the form $W'=D\cap \II$,
for some other split quaternion subalgebra $D'\subset \OO$. 
Thus, the $\mt{G}_2$-action on $\mt{Gr}(3,\II)$ 
has at least two orbits, one of which is $\mt{G}_2\cdot W_0$ 
and there is at least one other orbit of the form
$\mt{G}_2\cdot W$ for some 3-plane $W$ in $\II$.
The goal of our paper is to obtain an understanding 
of the geometry of the Zariski closure of the orbit 
$\mt{G}_2 \cdot W_0$,  which is the complexified 
version of  the real associative Grassmannian 
$\mt{G}_2^+/\mt{SO}_4(\R)$, where the  deformation 
theory of \cite{AS}  takes place. We achieve our goal 
by using techniques from calibrated geometries.

\vspace{.1in}

After we obtained some of our main results we learned from Michel Brion about the work of 
Alessandro Ruzzi~\cite{Ruzzi10,Ruzzi11} on the classification of symmetric varieties of Picard number 1. 
Our work fits nicely with this classification scheme, so we will briefly mention the relevant results of Ruzzi.

\vspace{.1in}

Let $G$ be a connected semisimple group defined over $\C$, 
$\theta$ be an involutory automorphism of $G$. We denote by
$G^\theta$ the fixed locus of $\theta$. Let $H$ be any subgroup 
that is squeezed between $(G^\theta)^0$ and the normalizer subgroup $N_G(G^\theta)$. 
Here, the superscript 0 indicates the connected component of the identity element.
The quotient varieties of the form $G/H$ are called symmetric varieties. 
In his 2011 paper, Ruzzi classified all symmetric varieties of Picard number 1 
and in ~\cite[Theorem 2a)]{Ruzzi10}, he showed that the smooth equivariant 
completion with Picard number 1 of the symmetric variety 
$\mt{G}_2/ \mt{SL}_2 \times \mt{SL}_2$ is the intersection of the grassmannian 
$\mt{Gr}(3,\II)$ with a 27 dimensional $\mt{G}_2$-stable linear space in $\PP(\bigwedge^3 \II)$. 
We will denote this equivariant completion by $X_{min}$ and 
call it the (complex) {\em associative grassmannian}.
Note that over $\C$, $\mt{SL}_2\times \mt{SL}_2$ is identified with the special 
orthogonal group $\mt{SO}_4$.
Although Ruzzi has first showed that $X_{min}$ is the unique smooth equivariant completion of 
$\mt{G}_2/ \mt{SL}_2 \times \mt{SL}_2$ with Picard number 1, the question of finding 
its defining ideal as well as the computation of its Poincar\'e polynomial remained 
unanswered. In a sense our article finishes this program. 
More precisely, we prove the following results: 


\begin{Theorem}\label{T:main1:intro}
As a subvariety of $\mt{Gr}(3,\II)$, the compactification $X_{min}$ 
of $\mt{G}_2/\mt{SO}_4$ is defined by the vanishing of the following seven 
linear forms in the Pl\"ucker coordinates of $\mt{Gr}(3,\II)$:
\begin{enumerate}
\item $p_{247} - p_{256} - p_{346} - p_{357}$,
\item $p_{156}  - p_{147} + p_{345} - p_{367}$,
\item $-p_{245} + p_{267} + p_{146} + p_{157}$, 
\item $p_{567} + p_{127} - p_{136} + p_{235}$, 
\item $-p_{126} - p_{467} - p_{137} - p_{234}$, 
\item $p_{457} + p_{125} + p_{134} - p_{237}$,
\item $p_{135} - p_{124} - p_{456} + p_{236}$.
\end{enumerate}
\end{Theorem}

\begin{Theorem}\label{T:main2:intro}
The Poincar\'e polynomial of $X_{min}$ is
$$
P_{X_{min}}(t^{1/2}) = 1+ t + 2t^2 + 2t^3 + 3t^4 + 2t^5 + 2 t^6+ t^7 + t^8.
$$
\end{Theorem}

To prove these results we analyze the natural action of the maximal torus of $\mt{G}_2$ on $X_{min}$. 
In particular, we determine the fixed points of the torus action and compute the Poincar\'e polynomial of 
$X_{min}$ by using the Bia{\l}ynicki-Birula decomposition. 

Let us emphasize once more that none of our results rely on Ruzzi's work but 
we use the techniques from calibrated geometries. In fact, over the field of real numbers, 
the analogous symmetric variety $\text{ASS}:=\mt{G}_2^+/\mt{SO}_4(\R)$ is already compact, and its geometry 
is well understood; its defining equations are also given by certain linear equations arising from 
the calibration form. In this article, we essentially lifted these observations to the complex setting. 
Finally, let us mention that the complete description of the ring structure of the $H^*(\text{ASS},\Z)$ is described in \cite{AK}.

\vspace{1cm}

\textbf{Acknowledgement.}  We thank Michel Brion for his comments on an earlier version of our paper and for bringing to our attention the work of Ruzzi. 
We thank \"Ust\"un Y\i ld\i r\i m and \"Ozlem U\u{g}urlu for their comments and help.
Finally, we thank to the anonymous referee for her/his careful reading of our paper and for 
her/his comments which improved the quality of our paper in a significant way.

\section{Grassmann of 3-planes}\label{S:grassmannian}

We call  a 3-plane $W \in \mt{Gr}(3,\II_k)$ associative if $W = D \cap \II_k$, 
where $D$ is a quaternion subalgebra of $\OO_k$.
For a subset $S\subset \OO_k$, we denote by $A(S)$ 
the subalgebra of $\OO_k$ that is generated by $S$. 
Let $W\in \mt{Gr}_k(3,\II_k)$ be a 3-plane and let 
$u_1,u_2,u_3\in W$ be a basis. Thus, the vector space dimension 
of $A(W)$ is either $\dim A(W) = 4$ or $\dim A(W)=8$.
In latter case, obviously, $A(W)=\OO_k$.
In the former case, $A(W)$ is an associative subalgebra 
(as follows from Proposition 1.5.2 of~\cite{SV})
but it does not need to be a composition subalgebra. 
Our orbit $\mt{G}_2 \cdot W_0$ in $\mt{Gr}_k(3,\II_k)$
contains the set of 3-planes $W$ such that $A(W)$ is a 
4 dimensional composition subalgebra, 
which we state in our next lemma:

\begin{Lemma}\label{L:fixed quaternions}
If $W\subset \II$ is a 3-plane that is in the 
$\mt{G}_2$-orbit of $W_0$, then $A(W)$ is 
a quaternion subalgebra of $\OO$. Moreover, 
the stabilizer subgroup of any such $W$ is 
isomorphic to $\mt{SO}_4$, the special 
orthogonal group of 4 by 4 matrices.
\end{Lemma}

\begin{proof}
The subalgebra generated by $W_0$ is the quaternion 
algebra $\DD$.
Since $\mt{G}_2$ acts transitively on the set of 
quaternion subalgebras (Corollary 2.2.4~\cite{SV}), 
the proof of our first assertion follows. 
To prove our second claim it is enough to prove it for $W_0$, 
the ``origin of the orbit'' since the stabilizer subgroups of 
other points are isomorphic to that of $W_0$ by conjugation.

Let $g$ be an element from $\mt{G}_2$ such that $g\cdot W_0 = W_0$.
Then $g$ acts on the orthogonal complement $\DD^\perp$. 
Recall that $\mt{G}_2$ is contained in $SO(\II_k)$. 
Therefore, on one hand we have an injection 
$\iota:\ \mt{Stab}_{\mt{G}_2} (W_0)\hookrightarrow SO(\DD^\perp) \simeq \mt{SO}_4$. 
On the other hand, we know that the elements of $SO(\DD^\perp)$ 
are completely determined by how they act on the
part of the basis $e_5,e_6,e_7,e_8$ of $\OO$. 
Indeed, we see this from Figure~\ref{F:FANO}, 
which gives us the multiplicative structure of $\OO$.

\begin{figure}[htp]
\begin{center}
\begin{tabular}{ l| |l |l |l |l ||l |l |l |l  | }
  
           & $e$ & $e_2$ & $e_3$ & $e_4$ & $e_5$ & $e_6$ & $e_7$ & $e_8$ \\   \hline \hline
$e$ & $e$ & $e_2$ & $e_3$ & $e_4$ & $e_5$ & $e_6$ & $e_7$ & $e_8$  \\   \hline
$e_2$ & $e_2$& $-e$ & $e_4$ & $-e_3$ & $e_6$ & $-e_5$ & $-e_8$ & $e_7$  \\   \hline
$e_3$ & $e_3$ & $-e_4$ & $-e$ & $e_2$ & $e_7$ & $e_8$ & $-e_5$ & $-e_6$  \\   \hline
$e_4$ & $e_4$ & $-e_3$ & $-e_2$ & $-e$ & $e_8$ & $-e_7$ & $e_6$ & $-e_5$  \\   \hline \hline

$e_5$ & $e_5$ & $-e_6$ & $-e_7$ & $-e_8$ & $e$ & $-e_2$ & $-e_3$ & $-e_4$  \\   \hline
$e_6$ & $e_6$ & $e_5$ & $-e_8$ & $e_7$ & $e_2$ & $e$ & $e_4$ & $-e_3$  \\   \hline
$e_7$ & $e_7$ & $e_8$ & $e_5$ & $-e_6$ & $e_3$ & $-e_4$ & $e$ & $e_2$  \\   \hline

$e_8$ & $e_8$ & $-e_7$ & $e_6$ & $e_5$ & $e_4$ & $e_3$ & $-e_2$ & $e$  \\   \hline

\end{tabular}
\caption{Multiplication table for split octonions.}
\label{F:FANO}
\end{center}
\end{figure}

The multiplication table of $e_5,e_6,e_7,e_8$ includes $e_2,e_3,e_4$ and $e$, therefore, the action of $g$ on $W$ is uniquely determined 
by the action of $g$ on $e_5,e_6,e_7,e_8$. It follows that $\iota$ is surjective as well, hence it is an isomorphism.
\end{proof}

\begin{Remark}
The element-wise stabilizer of $\DD$ in $G$ is isomorphic to $\mt{SL}_2$. (See Proposition 2.2.1~\cite{SV}).
Heuristically, this follows from the fact that $\OO=\DD \oplus \DD$, and that $(\DD,N) = (\mt{Mat}_2,\det)$.
\end{Remark}

Since $\OO= \DD \oplus \DD$ and $\DD = \mt{Mat}_2$, we take $\{ (e_{ij},0):\ i,j=1,2 \} \cup \{ (0,e_{ij}):\ i,j=1,2\}$ as a basis for $\OO$. 
Here, $e_{ij}$ is the $2\times 2$ matrix with 1 at the $i,j$th position and 0's everywhere else. 
Recall that $\II$ is the orthogonal complement of the identity $e= (e_{11}+e_{22},0)$ of $\OO$.
A straightforward computation shows that 
$(x,y)\in \OO$ is in $\II$ if and only if the trace of $x$ is 0.
Thus, we write $\II = \mf{sl}_2 \oplus \mt{Mat}_2$ (we are going to make use of Lie algebra structure on $\mf{sl}_2$ in the sequel).

\vspace{.5cm}

Let $W\in \mt{Gr}(3,\II)$ be a 3-plane in $\II$ and let $\{u_1,u_2,u_3\}$ be a basis for $W$. 
The map $P: \mt{Gr}(3,\II) \rightarrow \PP(\bigwedge^3 \II)$ defined by $P(W) = [u_1 \wedge u_2 \wedge u_3]$ is the Pl\"ucker embedding of 
$\mt{Gr}(3,\II)$ into the 34 dimensional projective space $\PP(\bigwedge^3 \II)$. 
Note that $\mt{GL}(\II)$ acts on both of the varieties $\mt{Gr}(3,\II)$ and $\bigwedge^3 \II$ via its natural action on $\II$. 
Note also that the Pl\"ucker embedding is equivariant with respect to these actions. 
In particular, it is equivariant with respect to the subgroup $\mt{G}_2$.

We make the identifications
\begin{align}
1 \leftrightarrow \begin{pmatrix}
  1 & 0 \\
  0 & 1
\end{pmatrix},\
\ii \leftrightarrow \begin{pmatrix}
  i & 0 \\
  0         & -i
\end{pmatrix},\
\jj \leftrightarrow \begin{pmatrix}
  0 & 1 \\
  -1 & 0
\end{pmatrix}, \
\kk \leftrightarrow \begin{pmatrix}
  0         & i \\
  i & 0
\end{pmatrix},
\end{align}
and take $\{ (\ii,0),(\jj,0),(\kk,0),(0,1),(0,\ii),(0,\jj),(0,\kk) \}$ as a basis for $\II$.
Let $W_0$ denote the span of $\{ (\ii,0),(\jj,0),(\kk,0)\}$ and let $W_0^*$ denote the span of $\{(0,\ii),(0,\jj),(0,\kk)\}$.
Thus,
\begin{align}
\II = W_0 \oplus W_0^* \oplus \C.
\end{align}
\begin{Remark}
It is noted earlier that a copy of $\mf{sl}_2$ sits in $\II$:
$$
\mf{sl}_2=\mf{sl}_2 \oplus 0 \hookrightarrow \mf{sl}_2 \oplus \mt{Mat}_2 = \II.
$$ 
This copy of $\mf{sl}_2$ is $W_0$ as a vector space.
\end{Remark}
\begin{Remark}
A straightforward calculation shows that if $(0,v)\in W_0^*$, then for all $(x,0)\in \mf{sl}_2$, 
$(x,0) (0,v) = (0,vx)$.
\end{Remark}

Next, we elaborate on a portion of the discussion from~\cite{FH}, \S 22.3 and analyze $\bigwedge^3 \II$ more closely.

Let $U$ denote $W_0 \oplus W_0^*$ so that we have 
$\bigwedge^3 (W_0 \oplus W_0^* \oplus \C ) = \bigoplus_{n=0}^3 \bigwedge^n U\otimes \bigwedge^{3-n} \C 
= \bigwedge^3 U \oplus \bigwedge^2 U$.
Since $\dim W_0 =\dim W_0^* = 3$, we have canonical identifications $W_0 = \bigwedge^{2} W_0^*$ and $W_0^* =\bigwedge^{2} W_0$.
It follows that 
\begin{align*}
\bigwedge^3 U &= (\C \otimes \C) \oplus (W_0 \otimes \bigwedge^2 W_0^*) \oplus (\bigwedge^2 W_0 \otimes W_0^*) \oplus (\C \otimes \C) \\
&= \C \oplus (W_0 \otimes W_0) \oplus (W_0^* \otimes W_0^*) \oplus  \C
\end{align*}
and that 
\begin{align*}
 \bigwedge^2 U &=  \bigwedge^2 W_0\otimes \C \oplus W_0 \otimes W_0^* \oplus \C \otimes \bigwedge^{2} W_0^*  \\
 &=  W_0^* \oplus (W_0 \otimes W_0^*) \oplus  W_0.
\end{align*}
Putting all of the above together we see that 
\begin{align*}
\bigwedge^3 \II &= \C \oplus (W_0 \otimes W_0) \oplus (W_0^* \otimes W_0^*) \oplus  \C \oplus W_0^* \oplus (W_0 \otimes W_0^*) \oplus  W_0 \\
 &= \II \oplus (W_0 \otimes W_0) \oplus (W_0^* \otimes W_0^*) \oplus  (W_0 \otimes W_0^*) \oplus \C.
\end{align*}

Next, we analyze $\mt{Sym}^2 \II$ more closely; 
\begin{align*}
\mt{Sym}^2 \II &= \mt{Sym}^2 ( U \oplus \C)\\ 
&=  (\mt{Sym}^2 U \otimes \C) \oplus (\mt{Sym}^1 U \otimes \C) \oplus (\C \otimes \mt{Sym}^2 \C) \\
&=  \mt{Sym}^2 W_0 \oplus (W_0 \otimes W_0^*) \oplus \mt{Sym}^2 W_0^* \oplus (W_0 \oplus W_0^*)\oplus \C  \\
&=  (\mt{Sym}^2 W_0 \oplus W_0^*) \oplus (W_0 \otimes W_0^*) \oplus (\mt{Sym}^2 W_0^* \oplus W_0) \oplus \C \\
&=  (\mt{Sym}^2 W_0 \oplus \bigwedge^2 W_0) \oplus (W_0 \otimes W_0^*) \oplus (\mt{Sym}^2 W_0^* \oplus \bigwedge^2 W_0^*) \oplus \C  \\
&=  \mt{End}(W_0) \oplus (W_0 \otimes W_0^*) \oplus \mt{End}(W_0^*) \oplus \C \\
&\simeq  (W_0\otimes W_0) \oplus (W_0 \otimes W_0^*) \oplus (W_0^*\otimes W_0^*) \oplus \C. 
\end{align*}
\begin{Remark}
The last term is only an isomorphism since we are using non-canonical identification of $W_0$ with $W_0^*$.
\end{Remark}
Therefore, we see that 
\begin{align}
\bigwedge^3 \II \simeq \mt{Sym}^2 \II \oplus \II.
\end{align}
Furthermore, it is true that $\mt{Sym}^2 \II = \Gamma_{2,0} \oplus \C$, where 
\begin{align}
\Gamma_{2,0} &= (W_0 \otimes W_0) \oplus (W_0^* \otimes W_0^*) \oplus  (W_0 \otimes W_0^*)  \notag \\
&= (W_0 \otimes \bigwedge^2 W_0^*) \oplus (\bigwedge^2 W_0 \otimes W_0^*)  \oplus  (W_0 \otimes W_0^*). \label{A:Gamma}
\end{align}
is an irreducible representation of $\mt{G}_2$ with highest weight $2\omega_1$, where $\omega_1$ is the highest weight 
of the first fundamental representation $\II$ of $\mt{G}_2$. (See~\cite{FH}, \S 22.3.)
Once the root system $\Phi = \{\alpha_1,\dots, \alpha_6,\beta_1,\dots, \beta_6\}$ is chosen as in~\cite{FH}, \S 22.2 (pg. 347), 
we see that $2\omega_1 = \alpha_1 + \alpha_3+\alpha_4$.

The structure of the representation of $\mt{G}_2$ on $\II$ 
can be spelled out to a finer degree once we linearize the action. 
Let $\mf{g}_2$ denote the Lie algebra of $\mt{G}_2$. 
It is well known that $\mf{g}_2$ contains a copy of $\mf{g}_0 = \mf{sl}_3$, and moreover, 
as a representation of $\mf{g}_0$ it has the following decomposition:
\begin{align*}
\mf{g}_2 = \mf{g}_0 \oplus W \oplus W^*,
\end{align*}
where $W$ is isomorphic to the standard 
3 dimensional representation $\C^3$ of $\mf{sl}_3$ (See~\cite{FH}, \S 22.2.). 
Furthermore, the unique 7 dimensional irreducible representation $V$ of $\mf{g}_2$
can be identified with $V= W\oplus W^* \oplus \C$ as an $\mf{sl}_3$-module. 
In our notation, we are going to take $W$ as $W_0$.
Before making this identification we choose a basis for $W$ using the root 
system $\Phi=\{\alpha_1,\dots, \beta_{12}\}$.

Let $V_i\subset V$ ($i=1,\dots, 6$) denote the eigenspace (corresponding to the eigenvalue $\alpha_i$) for the action of the maximal abelian 
subalgebra $\mf{h}\subset \mf{g}_2$ corresponding to $\Phi$. Let $Y_i$ be the root vector whose 
eigenvalue is $-\alpha_i=\beta_i$ for $i=1,\dots, 6$. 
Arguing as in pg. 354 of~\cite{FH}, we have the basis $e_1=v_4,e_2=w_1,e_3=w_3$ for $W$, where 
$w_i$'s are found as follows: 
\begin{align*}
v_3 &= Y_1(v_4),\\ 
v_1 &=-Y_2(v_3),\\ 
u &=Y_1(v_1),\\
w_1 &=\frac{1}{2}Y_1(u), \\ 
w_3 &= Y_2(w_1), \\ 
w_4&=-Y_1(w_3). 
\end{align*}
The corresponding dual basis elements $e_1^*,e_2^*,e_3^*$ are given by $w_4,v_1,v_3$, respectively.
The upshot of all of these is that we identify $W_0$ with $W$ in such a way that the basis $v_4,w_1,w_3$ 
corresponds (in the given order) to $(\ii,0),(\jj,0),(\kk,0)$, and the basis $w_4,v_1,v_3$ for $W^*$ corresponds to 
$(0,\ii),(0,\jj),(0,\kk)$. 
With respect to these identifications, we observe that the highest weight vector in $\mt{Sym}^2 V \subset \bigwedge^3 V$ 
of the highest weight $2\omega_1=\alpha_1+\alpha_3+\alpha_4$ is given by the 3-form $v_1 \wedge v_3 \wedge v_ 4$, or by
$ (\ii,0) \wedge  (0,\jj) \wedge (0,\kk)$ in the case of $\mt{Sym}^2 \II \subset \bigwedge^3 \II$. 
It is clear that the 3-form $ (\ii,0) \wedge  (0,\jj) \wedge (0,\kk)$ is actually an element of $W_0\otimes \bigwedge^2 W_0^* \subset \Gamma_{2,0}$ 
by~(\ref{A:Gamma}).

It is straightforward to verify that the octonions $ (\ii,0),  (0,\jj),$ and $(0,\kk)$ generate a quaternion algebra
which we denote by $\UU$. By Lemma~\ref{L:fixed quaternions} we see that the stabilizer subgroup of $\UU$ is $\mt{SO}_4$. 
It is well known that the highest weight vector in $\Gamma_{2,0} \subset \mt{Sym}^2 (\II)$ is the direction vector of 
the line that is stabilized by $\mt{SO}_4$.

We view the projectivization $\PP(\Gamma_{2,0})$ as a (closed) subvariety of $\PP(\bigwedge^3 \II)$. 
The image of $\mt{Gr}(3,\II)$ intersects $\PP(\Gamma_{2,0})$. 
In fact, by the above discussion we know that the image of the 3-plane $U_0 := \UU \cap \II \in \mt{Gr}(3,\II)$ under 
Pl\"ucker embedding is the $\mt{SO}_4$-fixed point $[u_0] \in \PP(\Gamma_{2,0})$.
On one hand, since it is a $\mt{G}_2$-equivariant isomorphism onto its image, the orbit $\mt{G}_2\cdot U_0$ in $\mt{Gr}(3,\II)$ is mapped 
isomorphically onto $\mt{G}_2\cdot [u_0]$ in $\PP(\Gamma_{2,0}) \subset \PP(\bigwedge^3 \II)$. 
On the other hand, as we are going to see in the sequel, the closure of the orbit $\mt{G}_2\cdot U_0$ in $\mt{Gr}(3,\II)$ is smooth,
however, the Zariski closure of the orbit $\mt{G}_2\cdot [u_0]$ in $\PP(\bigwedge^3 \II)$ is not. 
The latter closure is the smallest ``degenerate'' compactification of the symmetric variety $\mt{G}_2/\mt{SO}_4$, whereas the former 
compactification is the smallest, smooth $\mt{G}_2$-equivariant compactification.

\section{More on Octonions}

In this section we collect and improve some known facts about alternating 
forms on (split) composition algebras.

The multiplicative structure of a quaternion algebra (over $\R$ or $\C$) is always associative (but not commutative). 
To measure how badly the associativity of multiplication fails in $\OO$ one looks at the {\em associator}, defined by
\begin{align}\label{A:associator}
[x,y,z]  = (xy)z - x(yz) \text{ for all } x,y,z\in \OO_k.
\end{align}
It is well known that the associator is an alternating 3-form (see Section 1.4 of~\cite{SV}).

There are several other related multiplication laws on the imaginary part $\II_k$ of $\OO_k$.
For example, the ``cross-product'' is defined by 
$$
a\times b =\frac{1}{2}  ( ab- ba) \text{ for all } a,b\in \II_k.
$$
Obviously, the cross-product is alternating.

The ``dot product'' is defined by 
$$
a \cdot b = -\frac{1}{2} (ab + ba) \text{ for all } a,b\in \II_k.
$$
It is also obvious that $ab = a\times b - a\cdot b$ for $a,b\in \II_k$.

These products are easily extended to $\OO_k$. Indeed, any element of $\OO_k$ has the form $x=\alpha (1,0) + a$, 
where $\alpha \in k$, $a\in \II_k$, and if $y=\beta (1,0) + b$ is another element from $\OO_k$, then 
\begin{align}\label{A:like Zorn}
xy= (\alpha,a)(\beta,b) = (\alpha \beta - a\cdot b, \alpha a +\beta b + a \times b ),
\end{align}
where we use the identification $\alpha (1,0) + a = (\alpha, a)$.
In particular, if $x=a$ and $y=b$ are from $\II_k$, then $xy =(-x\cdot y,x\times y)$, hence 
\begin{align}\label{A:norm of cross product}
N(xy)= (x\cdot y)^2 + N(x\times y).
\end{align}

Now we focus on $k=\C$ and extend some results from~\cite{HL} to our setting. 
First, we re-label the basis $\{e,e_1,\dots, e_7\}$ for $\OO$ so that 
$\{e_1 = (\ii,0),e_2=(\jj,0),e_3=(\kk,0),e_4=(0,1),e_5=(0,\ii),e_6 =(0,\jj),e_7=(0,\kk)\}$
is the standard basis for $\II$. Consider the trilinear form 
\begin{align*}
\varphi (x,y,z) = \langle x \times y,z \rangle,\qquad x,y,z\in \II.
\end{align*}
\begin{Lemma}
$\varphi$ is an alternating 3-form on $\II$.
\end{Lemma}
\begin{proof}
Since both cross-product and the inner product $\langle, \rangle$ are bilinear, we see it is enough to 
check the assertion on the basis $\{e_1,\dots, e_7\}$. We verified this by using software called Maple.
\end{proof}
\begin{Remark}
It follows from (\ref{A:like Zorn}) that 
\begin{align}\label{L:3form}
\varphi(x,y,z) = \langle xy , z \rangle \ \text{ for } x,y,z\in \II.
\end{align}
It is not difficult to verify (by using Maple, or by hand) that $\varphi$ is equal to the 3-form 
\begin{align}\label{A:phi}
\varphi= e^{123} - e^{145} + e^{167} -e^{246} -e^{257} - e^{347} + e^{356},
\end{align}
where $e^{ijk}=de_i\wedge de_j \wedge de_k$ as before. 
In particular, we see that from (\ref{L:3form}) that $\mt{G}_2=\mt{Aut}(\OO)$ stabilizes the form (\ref{A:phi}).
\end{Remark}

\begin{Definition}
For a 3-plane $W\in \mt{Gr}(3,\II)$ we define $\varphi(W)$ to be the evaluation of $\varphi$ on 
any orthonormal basis $\{x,y,z\}$ of $W$. 
\end{Definition}

\begin{Theorem}\label{T:generates iff}
If a 3-plane $W\in \mt{Gr}(3,\II)$ is associative, 
then $\varphi(x,y,z)\in \{-1,+1\}$ for any orthonormal basis $\{x,y,z\}$ of $W$.    
\end{Theorem}

\begin{proof}
Any two elements $x,y$ of an orthonormal triplet $(x,y,z)$ from $\II$ form a ``special $(1,1)$-pair'' in the sense 
of~\cite{SV}, Definition 1.7.4.\footnote{A pair $(x,y)$ of elements from a composition algebra is called a special 
$(\lambda,\mu)$-pair if $ \langle x, e \rangle=\langle y, e \rangle=\langle x, y \rangle=0 $, $N(x)=\lambda$ and $N(y)=\mu$.}
Since $\mt{G}_2$ acts transitively on special $(1,1)$-pairs, and since $((\ii,0),(\jj,0))$ is such, there exists $g\in \mt{G}_2$ such that 
$g(x) = (\ii,0)$, $g(y)=(\jj,0)$.  In particular, it follows that $g(xy)= (\kk,0)$. 

We claim that if $x,y,z$ generates a quaternion algebra, then $g(z) = \pm (\kk,0)$. 
Indeed, unless $xy$ is a scalar multiple of $z$, the span in $\II$ of $e,x,y,z$ and $xy$ is 5 dimensional, hence 
the composition algebra generated by $x,y,z$ is not a quaternion subalgebra. 
It follows that, if $\{1,x,y,z\}$ is an orthonormal basis for a quaternion subalgebra, then $z$ is a scalar multiple of $xy$.
Since the norm of $z$ is 1, we see that $z = \pm xy$, hence $g(z)=\pm(\kk,0)$. Then 
\begin{align*}
\langle xy, z \rangle &= \langle g^{-1} ((\ii,0)) g^{-1}((\jj,0)), g^{-1}(\pm (\kk,0))\rangle \\
&= \langle g^{-1} ((\ii,0)(\jj,0)), g^{-1}(\pm (\kk,0))\rangle \\
&= \langle (\ii,0)(\jj,0),\pm(\kk,0)\rangle \\
&= \pm 1.
\end{align*}
\end{proof}

We define the {\em triple-cross product} on $\OO$ as follows  
\begin{align}
x \times y \times z = \frac{1}{2}( x(\bar y z) - z (\bar yx)) \ \text{ for all } x,y,z\in \OO.
\end{align}

\begin{Lemma}\label{L:trilinear and alternating}
The triple-cross product is trilinear and alternating. Moreover, $N(x\times y\times z) = N(x)N(y) N(z)$ for all 
$x,y,z\in \OO$ distinct from each other.
\end{Lemma}
\begin{proof}
The trilinearity is obvious. To prove the second claim we check $x\times x \times z=0$, $x \times y \times y=0$, and $x\times y \times x=0$.
We use \cite{SV}, Lemma 1.3.3 $i)$ for the first two:
\begin{align*}
x\times x \times z &=  \frac{1}{2}( x(\bar x z) - z (\bar xx)) =  \frac{1}{2}( N(x)z - z N(x)) =0, \\
x\times y \times y &=  \frac{1}{2}( x(\bar y y) - y (\bar yx)) =  \frac{1}{2}( x N(y) -  N(y)x) =0,\\
x\times y \times x &=  \frac{1}{2}( x(\bar y x) - x (\bar yx)) =0.
\end{align*}
Finally, to prove $N(x\times y\times z) = N(x)N(y) N(z)$ we expand $x\times y \times z$ in the orthonormal basis
$e=(1,0),e_1,\dots, e_7$ of $\OO$. In particular, this allows us to assume that $x,y$ and $z$ as scalar multiples of standard basis vectors. 
Now it is straightforward to verify (by Maple) that $x(\bar y z ) =- z (\bar y x)$, hence $x\times y \times z = x (\bar y z )$ and our claim follows
by taking norms. 
\end{proof}


Next, we show that the ``associator identity'' of Harvey-Lawson (Theorem 1.6~\cite{HL}) holds for split octonion algebras.
\begin{Theorem}\label{T:associator identity}
For all $x,y,z\in \II$, the associator $[x,y,z]$ lies in $\II$, and 
\begin{align}\label{A:ass 1}
x\times y \times z= \langle xy,z\rangle e+  [ x,y,z],
\end{align}
where $e$ is the identity element $(1,0)$ of $\OO$.
Moreover, 
\begin{align}\label{A:ass 2}
| x\wedge  y \wedge z | = \varphi(x,y,z)^2 + N([x,y,z]),
\end{align}
where $ | x\wedge  y \wedge z |$ is defined as $N(x)N(y)N(z)$.
\end{Theorem}
\begin{proof}
By linearity and alternating property, it is enough to prove our first two assertions for (orthonormal) triplets $(x,y,z)$ from the basis $\{e_1,\dots, e_7\}$
and once again this is straightforward to verify by using Maple. 
To prove (\ref{A:ass 2}), we note as in the proof of Lemma~\ref{L:trilinear and alternating} that 
$N(x\times y \times z) = N(x (\bar y z))=N(x)N(y)N(z)= | x \wedge y \wedge z |$.
Thus, our claim now follows from the simple fact that $N(\alpha e + u ) = \alpha^2 +N(u)$ whenever $u\in e^\perp, \alpha \in \C$.
\end{proof}

\begin{Corollary}
Let $W\in \mt{Gr}(3,\II)$ be a 3-plane and let $x,y,z$ be an orthonormal basis for $W$ (which always exists by
Gram-Schmidt process). If $W$ is associative, then $[x,y,z] = 0$. 
\end{Corollary}

\begin{proof}
By linearity and alternating property, it is enough to prove the statement ``$\varphi(x,y,z)\in \{-1,+1\}$ if and only if $[x,y,z]=0$''
on the orthonormal basis $\{e_1,\dots, e_7\}$. We verified the cases by using Maple. 
The rest of the proof follows from Theorem~\ref{T:generates iff}.
\end{proof}

Recall that when $x,y$, and $z$ are orthonormal vectors from $\II$, $x\times y\times z = x (\bar y z)$. 
We verified by using Maple that if the associator $[x,y,z]$ is nonzero for $\{x,y,z\} \subset \{e_1,\dots, e_7\}$, then 
\begin{align}\label{A:nice}
x\times y \times z = [x,y,z] \ \text{ and } \ \varphi(x,y,z)=0,
\end{align}
and if $[x,y,z]=0$, then 
\begin{align}\label{A:nice2}
x\times y \times z =  \varphi(x,y,z)e,
\end{align}
which conforms with (\ref{A:ass 1}).
It is not difficult to show that the form $*\varphi (u,v,w,z)$ defined by 
$$
*\varphi (u,v,w,z) = \langle u\times v \times w, z \rangle
$$ 
is an alternating 4-form, and it can be expressed as in
\begin{align}\label{A:4-form}
*\varphi = -e^{4567} + e^{2367} - e^{2345} + e^{1357} + e^{1346} + e^{1256} - e^{1247}.
\end{align}
This can be seen from the fact that the monomials of $\varphi$ and $*\varphi$ are complementary in the sense that 
$e^{ijkl}$ appears in $*\varphi$ if any only if there exists unique monomial $e^{rst}$ in $\varphi$
such that $\{r,s,t,i,j,k,l\} = \{1,2,3,4,5,6,7\}$. Also, it can be checked directly on the standard orthogonal basis.
Now, we define a new 3-form $\chi(u,v,w)$ by the identity $\langle \chi ( u,v,w), z \rangle = *\varphi(u,v,w,z)$.

An important consequence of these definitions and the above discussion (specifically, the equation (\ref{A:nice})) is that 
if a 3-plane spanned by $x,y,z\in \II$ generates a quaternion subalgebra, then $[x,y,z]=0$, which implies $\chi(x,y,z) = 0$. 
We record this in our next lemma:
\begin{Lemma}\label{L:chi vanishes}
If the 3-plane $W$ generated by $x,y,z\in \II$ is associative, then $\chi(x,y,z)=0$.  
\end{Lemma}

The vanishing locus of $\chi$ on $\PP(\bigwedge^3 \II)$ can be made more precise since
\begin{align*}
\chi &= (e^{247} -e^{256} - e^{346} - e^{357})e_1 \\ 
&+ ( e^{156} - e^{147} + e^{345} - e^{367})e_2 \\
&+ (-e^{245} + e^{267} + e^{146} + e^{157})e_3 \\
&+ (e^{567} + e^{127} - e^{136} + e^{235})e_4\\
&+ (-e^{126} - e^{467} - e^{137} - e^{234})e_5\\ 
&+ (e^{457} +e^{125} +e^{134} - e^{237} )e_6\\
&+ (e^{135} - e^{124} - e^{456} + e^{236})e_7,
\end{align*}
which follows from (\ref{A:4-form}). 

\begin{Remark}
It appears that the idea of using these seven linear equations obtained from $\chi$ to study associative manifolds is first used in \cite{AS10}.
\end{Remark}

\begin{Remark}\label{R:already}

Let $p_I$ ($I$ is a $d$-element subset of $\{1,\dots, n\}$) 
denote the (Pl\"ucker) coordinates on $\PP(\bigwedge^d \C^n)$.
The homogenous coordinate ring of the Grassmann 
variety of $d$ dimensional subspaces in $\C^n$ 
is the quotient of the polynomial ring 
$\C[p_I:\ I\subset \{1,\dots, n\}, \ |I|=d]$ by the ideal g
enerated by the following quadratic polynomials: 
$\sum_{s=1}^{d+1} (-1)^s p_{i_1 i_2 \dots i_{d-1} j_s } p_{j_1 j_2 \dots \widehat{j_s} \dots j_{d+1}}$,
where $i_1,\dots, i_{d-1},j_1,\dots, j_{d+1}$ are arbitrary numbers from 
$\{1,\dots, n\}$. Here, the hatted entry $\widehat{j_s}$ is omitted from the sequence. 
Of course, the case of interest for us is when $d=3$, $n=7$, and (\ref{A:Plucker relations}) is a (at most) 4-term relation:
\begin{align}\label{A:Plucker relations}
p_{i_1i_2 j_1} p_{j_2j_3 j_4} = p_{i_1i_2 j_2} p_{j_1j_3 j_4} -p_{i_1i_2 j_3} p_{j_1j_2 j_4} +  p_{i_1i_2 j_4} p_{j_1j_2j_3}.
\end{align}
In this notation, the vanishing locus of $\chi$ on $\PP(\bigwedge^3 \II)$ 
can be expressed in Pl\"ucker coordinates on $\PP(\bigwedge^d \C^n)$ by the following 7 linear equations: 
\begin{align}
p_{247} - p_{256} - p_{346} - p_{357} &= 0 \label{A:linear equations 1}\\ 
p_{156}  - p_{147} + p_{345} - p_{367} &= 0 \label{A:linear equations 2}\\ 
-p_{245} + p_{267} + p_{146} + p_{157} &= 0 \label{A:linear equations 3}\\ 
p_{567} + p_{127} - p_{136} + p_{235} &= 0 \label{A:linear equations 4}\\ 
-p_{126} - p_{467} - p_{137} - p_{234} &= 0 \label{A:linear equations 5}\\ 
p_{457} + p_{125} + p_{134} - p_{237} &= 0 \label{A:linear equations 6}\\ 
p_{135} - p_{124} - p_{456} + p_{236} &= 0 \label{A:linear equations 7}.
\end{align}
Moreover, it follows from Lemma~\ref{L:chi vanishes} that the Zariski closure of 
the space of associative 3-planes in $\PP(\bigwedge^3 \II)$, 
namely the image of the associative grassmannian $\mt{G}_2/\mt{SO}_4$ under the Pl\"ucker embedding of $\mt{Gr} (3,\II)$ 
lies in the intersection of these 7 hyperplanes with $\mt{Gr}(3,\II)$.
\end{Remark}

\begin{Definition}
We denote by $X_{min}$ the intersection in $\PP(\bigwedge^3 \II)$ of the 7 hyperplanes 
(\ref{A:linear equations 1})--(\ref{A:linear equations 7}) with the grassmannian $\mt{Gr} (3,\II)$.
\end{Definition}

\section{Two $\mt{SL}_2$ actions}\label{S:two SLs}

A 2-dimensional maximal torus $T$ of $\mt{G}_2$ is described by Springer and Veldkamp in Section 2.3 of~\cite{SV}
as the subgroup of automorphisms of $\OO$ consisting of the following transformations:
\begin{align*}
t_{\lambda,\mu}:\ (x,y) \mapsto (c_\lambda x c_\lambda^{-1}, c_\mu y c_\lambda^{-1}),
\end{align*}
where $(x,y)\in \OO$, $\lambda,\mu \in k^*$ and $c_\lambda,c_\mu$ are the diagonal matrices $\mt{diag}(\lambda, \lambda^{-1})$,
$\mt{diag}(\mu,\mu^{-1})$, respectively.

We look more closely at how $T$ acts on the grassmannian, so we express the action in our coordinates. 
Let $e_{ij}$ denote the elementary $2\times 2$ matrix which has 1 at $i,j$'th position and 0's elsewhere. 
The set of pairs $\{ (e_{11},0),(e_{12},0),(e_{21},0),(e_{22},0),(0,e_{11}),(0,e_{12}),(0,e_{21}),(0,e_{22}) \}$ forms a basis for $\OO$.
In this basis, $t_{\lambda,\mu}$ is the diagonal matrix
$$
t_{\lambda,\mu} = \mt{diag}(1,\lambda^2, \lambda^{-2}, 1, \lambda^{-1}\mu, \lambda \mu, \lambda^{-1} \mu^{-1}, \lambda \mu^{-1} ).
$$
We are going to switch to the basis 
$\{e=(e_{11}+e_{22},0), e_1 = (\ii,0),e_2=(\jj,0),e_3=(\kk,0),e_4=(0,e_{11}+e_{22}),e_5=(0,\ii),e_6 =(0,\jj),e_7=(0,\kk)\}$.
The proof of the next lemma is straightforward so we skip it.
\begin{Lemma}\label{L:torus action on e basis}
The action of maximal torus $T=t_{\lambda,\mu}$ of $\mt{G}_2$ on the basis $\{e_1,\dots, e_7\}$ of $\II$ is given by 
\begin{align*}
t_{\lambda ,\mu} (e) &= e\\
t_{\lambda ,\mu} (e_1) &= e_1 \\
t_{\lambda ,\mu} (e_2) &=  \left( \frac{\lambda^2 + \lambda^{-2}}{2} \right) e_2 + i\left( \frac{-\lambda^2 + \lambda^{-2}}{2} \right) e_3 \\  
t_{\lambda ,\mu} (e_3) &=  i\left( \frac{\lambda^2 - \lambda^{-2}}{2} \right) e_2 + \left( \frac{\lambda^2 + \lambda^{-2}}{2} \right) e_3 \\  
t_{\lambda ,\mu} (e_4) &= \left(\frac{ \lambda^{-1} \mu + \lambda \mu^{-1}}{2} \right) e_4 +i\left(\frac{ -\lambda^{-1} \mu + \lambda \mu^{-1}}{2} \right) e_5 \\ 
t_{\lambda ,\mu} (e_5) &= i\left(\frac{ \lambda^{-1} \mu - \lambda \mu^{-1}}{2} \right) e_4 +\left(\frac{ \lambda^{-1} \mu + \lambda \mu^{-1}}{2} \right) e_5 \\ 
t_{\lambda ,\mu} (e_6) &= \left(\frac{ \lambda \mu + \lambda^{-1} \mu^{-1}}{2} \right) e_6 + i \left(\frac{- \lambda \mu + \lambda^{-1} \mu^{-1}}{2} \right) e_7 \\ 
t_{\lambda ,\mu} (e_7) &= i\left(\frac{ \lambda \mu - \lambda^{-1} \mu^{-1}}{2} \right) e_6 +\left(\frac{ \lambda \mu + \lambda^{-1} \mu^{-1}}{2} \right) e_7.
\end{align*}
\end{Lemma}

There are two $\mt{SL}_2$'s naturally associated with the tori $t_{\lambda,\lambda}$ and $t_{\mt{id}_2,\mu}$.

\begin{Proposition}\label{P:SL2 action}
Let $x=(x_1,x_2)$ be an octonion from $\II = \mf{sl}_2 \oplus \mt{Mat}_2$. 
The two $\mt{SL}_2$ actions on $\II$ defined by 
\begin{enumerate}
\item $g\cdot x = (g x_1g^{-1}, gx_2 g^{-1})$ and 
\item $g\cdot x = (x_1, gx_2)$ 
\end{enumerate}
induce $\mt{SL}_2$ actions on associative 3-planes. 
\end{Proposition}
\begin{proof}
Let $W \in \mt{Gr}(3,\II)$ be an associative 3-plane spanned by the orthogonal basis 
$\{x,y,z\} \subset \II = \mf{sl}_2 \oplus \mt{Mat}_2$.
We know from Lemma~\ref{L:chi vanishes} that $W$ is associative if $[x,y,z] = 0$. 
Thus, it suffices to check the vanishing of the associator
$$
[g\cdot x, g\cdot y, g \cdot z] = (g\cdot x \, g\cdot y) g\cdot z -g \cdot x (g \cdot y\,  g \cdot z).
$$
Note that $\overline{g} = g^{-1}$ for all $g\in \mt{SL}_2$. 
Note also that for any $x=(x_1,x_2),y=(y_1,y_2)$ from $\II = \mf{sl}_2 \oplus \mt{Mat}_2$ we have 
\begin{align*}
(g\cdot x)(g\cdot y) &= (g x_1 y_1g^{-1} + \overline{g y_2g^{-1} }g x_2 g^{-1}, gy_2x_1g^{-1} + gx_2g^{-1}\overline{g y_1 g^{-1}})\\ 
&= (g x_1 y_1g^{-1} + g \overline{y_2} g^{-1}g x_2 g^{-1}, gy_2x_1g^{-1} + gx_2g^{-1}g \overline{y_1} g^{-1})\\
&= (g x_1 y_1g^{-1} + g \overline{y_2} x_2 g^{-1}, gy_2x_1g^{-1} + gx_2\overline{y_1} g^{-1})\\
&= g \cdot ( (x_1,x_2) (y_1,y_2) ). 
\end{align*}
Therefore, if $[x,y,z]=0$, then 
\begin{align*}
[g\cdot x, g\cdot y, g \cdot z] = (g\cdot x \, g\cdot y) g\cdot z -g \cdot x (g \cdot y\,  g \cdot z) &= (g \cdot (xy)) g\cdot z - g\cdot x (g\cdot (yz))\\
&= g \cdot ((xy)z) - g\cdot (x (yz))\\
&= g \cdot ((xy)z - x (yz)) \\
&= g \cdot [x,y,z]\\
&=0.
\end{align*}
Next, we check our claim for the second action:
\begin{align*}
(g\cdot x)(g\cdot y) &= (x_1 y_1 + \overline{g y_2}g x_2, gy_2x_1 + gx_2\overline{ y_1})\\ 
&= (x_1 y_1 + \overline{y_2} x_2, g(y_2x_1 + x_2\overline{ y_1}))\\ 
&= g \cdot ( (x_1,x_2) (y_1,y_2) ). 
\end{align*}
The rest follows as in the previous case.
\end{proof}

As a consequence of Proposition~\ref{P:SL2 action} we obtain two $\mt{SL}_2$ actions on 
$X_{min}$.

\begin{Remark}

We denote by $U$ the following unipotent subgroup:
$$
U = \left\{
\begin{pmatrix}
1& u \\
0 &1
\end{pmatrix}:\ u\in \C 
\right\} \subset \mt{SL}_2.
$$
The matrices of the actions of a generic element $g_u:=\begin{pmatrix}
1& u \\
0 &1
\end{pmatrix}\in U$
on the ordered basis $e_1,\dots, e_7$ of $\II$ are given by 
\begin{enumerate}
\item $
[g_u]=
\begin{pmatrix} 
1&-iu&-u&0&0&0&0\\ 
iu&1/2\,{u}^{2}+1&-i/2{u}^{2}&0&0&0&0\\ 
u&-i/2{u}^{2} &1-1/2\,{u}^{2}&0&0&0&0\\ 
0&0&0&1&0&0&0\\ 
0&0&0&0&1&-iu&-u\\ 
0&0&0&0&iu&1/2\,{u}^{2}+1&-i/2{u}^{2}\\
0&0&0&0&u&-i/2{u}^{2}&1-1/2\,{u}^{2}
\end{pmatrix}
$
\item 
$[g_u]=\begin{pmatrix} 
1&0&0&0&0&0&0\\ 
0&1&0&0&0&0&0\\ 
0&0&1&0&0&0&0\\ 
0&0&0&1&0& u/2&-i/2u\\
0&0&0&0 &1&-i/2u&-u/2\\ 
0&0&0&-u/2&i/2u&1&0\\ 
0&0&0&i/2u&u/2&0&1
\end{pmatrix}.
$
\end{enumerate}
For both of these actions of $U$ on $X_{min}$ the fixed point sets are positive dimensional. 
Indeed, the points $[e_{123}]$ and $[-e_{{126}}+ie_{{127}}+ie_{{136}}+e_{{137}}]$ of $\PP(\bigwedge^3 \II)$
lie on $X_{min}$ (this can be verified by using eqs. (\ref{A:linear equations 1})--(\ref{A:linear equations 7})) and both of these points  
are fixed by the first action of $U$. By a result of Horrocks~\cite{H}, we know that the fixed point set of a unipotent group acting on a connected complete 
variety is connected. Therefore, the fixed point set of $U$ on $X_{min}$ is positive dimensional for the first action.
Similarly, the points $[e_{123}]$ and $[-e_{{346}}+ie_{{347}}-ie_{{356}}+e_{{357}}]$ of $X_{min}$ are fixed by the second action of $U$,
hence the fixed point set of this action is also positive dimensional.

\end{Remark}


\section{Torus fixed points}\label{S:Torus fixed points}

Now we go back to analyzing fixed point set of the maximal torus $t_{\lambda,\mu}$ of $\mt{G}_2$ on $X_{min}$. 
The action of $t_{\lambda,\mu}$ on the basis $\{e_1,\dots, e_7\}$ is computed in the previous section.
The eigenvalues are $1, \frac{1}{\lambda^2},\lambda^2,\frac{\lambda}{\mu},\frac{\mu}{\lambda},\frac{1}{\lambda \mu},\lambda \mu$ 
and the respective eigenvectors are 
\begin{align*}
\widetilde{e_1}&= e_1,\\
\widetilde{e_2} &= -ie_2 + e_3,\\
\widetilde{e_3}&= ie_2 + e_3,\\
\widetilde{e_4} &= -ie_4 + e_5,\\
\widetilde{e_5} &= ie_4 + e_5,\\
\widetilde{e_6} &= -ie_6+e_7,\\
\widetilde{e_7} &= ie_6+e_7.
\end{align*}

For $i,j$ and $k$ from $\{1,\dots, 7\}$ we write $\wil{e}_{i j k}$ for $\wil{e}_i \wedge \wil{e}_j \wedge \wil{e}_k$. 
Accordingly, we write $\wil{p}_{ijk}$ for the transformed Pl\"ucker coordinate functions so that 
$$
\wil{p}_{ijk} (\wil{e}_{rst}) = 
\begin{cases}
1 & \text{ if } i =r,\ j= s,\ k=t; \\
0 & \text{ otherwise.}
\end{cases}
$$
Our defining equations (\ref{A:linear equations 1})--(\ref{A:linear equations 7}) become: 
\begin{align}
\wil{f}_1 &:= \wil{p}_{{247}} + \wil{p}_{{356}} = 0 \label{A:new 7a}  \\
\wil{f}_2 &:= 2\wil{p}_{{147}}+2\wil{p}_{{156}}+\wil{p}_{{245}}+\wil{p}_{{345}}-\wil{p}_{{267}}-\wil{p}_{{367}} = 0 \label{A:new 7b}\\
\wil{f}_3 &:= 2\wil{p}_{{147}}+2\wil{p}_{{156}}+\wil{p}_{{245}}-\wil{p}_{{345}}-\wil{p}_{{267}}+\wil{p}_{{367}} = 0 \label{A:new 7c}\\
\wil{f}_4 &:= 2\wil{p}_{{127}}-2\wil{p}_{{136}}+\wil{p}_{{234}}+\wil{p}_{{235}}+\wil{p}_{{467}}+\wil{p}_{{567}} = 0 \label{A:new 7d}\\
\wil{f}_5 &:= -2\wil{p}_{{127}}-2\wil{p}_{{136}}+\wil{p}_{{234}}-\wil{p}_{{235}}+\wil{p}_{{467}}-\wil{p}_{{567}} = 0 \label{A:new 7f}\\
\wil{f}_6 &:= 2\wil{p}_{{124}}-2\wil{p}_{{135}}-\wil{p}_{{236}}-\wil{p}_{{237}}+\wil{p}_{{456}}+\wil{p}_{{457}}  = 0 \label{A:new 7g}\\
\wil{f}_7 &:= 2\wil{p}_{{124}}+2\wil{p}_{{135}}-\wil{p}_{{236}}+\wil{p}_{{237}}+\wil{p}_{{456}}-\wil{p}_{{457}} = 0 \label{A:new 7h}.
\end{align}

It is easily verified that the following 35 vectors are eigenvectors for the action of $t_{\lambda,\mu}$ on $\bigwedge^3 \II$
(together with the eigenvalues indicated on the left column):
\begin{center}
\begin{tabular}{c|c}
${\frac {1}{{\lambda}^{3}\mu}}$ & $\wil{e}_{{126}}$ 
\\[.1cm]
\hline
${\lambda}^{3}\mu$ & $\wil{e}_{{137}}$  
\\[.1cm]
\hline
${\mu}^{2}$ & $\wil{e}_{{157}}$  
\\[.1cm]
\hline
${\mu}^{-2}$ & $\wil{e}_{{146}}$ 
\\[.1cm]
\hline
${\frac {{\lambda}^{2}}{{\mu}^{2}}}$ & $\wil{e}_{{346}}$ 
\\[.1cm]
\hline
${\frac {{\mu}^{2}}{{\lambda}^{2}}}$ & $\wil{e}_{{257}}$ 
\\[.1cm]
\hline
${\frac {{\lambda}^{3}}{\mu}}$ & $\wil{e}_{{134}}$ 
\\[.1cm]
\hline
${\frac {\mu}{{\lambda}^{3}}}$ & $\wil{e}_{{125}} $ 
\\[.1cm]
\hline
${\frac {1}{{\lambda}^{2}{\mu}^{2}}}$ & $\wil{e}_{{246}}$ 
\\[.1cm]
\hline
$ {\lambda}^{2}{\mu}^{2} $ & $\wil{e}_{{357}} $ 
\\[.1cm]
\hline
${\lambda}^{4} $ & $\wil{e}_{{347}}$  
\\[.1cm]
\hline
${\lambda}^{-4}$ & $\wil{e}_{{256}}$
\\[.1cm]
\hline
${\frac {\mu}{\lambda}}$ & $\wil{e}_{{567}},\ \wil{e}_{{235}},\ \wil{e}_{{127}}$ 
\\[.1cm]
\hline
${\frac {\lambda}{\mu}}$ & $\wil{e}_{{467}}, \ \wil{e}_{{234}},\ \wil{e}_{{136}}$ 
\\[.1cm]
\hline
${\frac {1}{\lambda\,\mu}}$ & $ \wil{e}_{{456}},\ \wil{e}_{{236}} ,\ \wil{e}_{{124}}$
\\[.1cm]
\hline
$\lambda \mu$ & $\wil{e}_{{457}} ,\ \wil{e}_{{237}},\ \wil{e}_{{135}}$ 
\\[.1cm]
\hline
$\frac{1}{\lambda^{2}} $ & $\wil{e}_{{267}} , \ \wil{e}_{{245}} ,\ \wil{e}_{{156}}$
\\[.1cm]
\hline
$ {\lambda}^{2}$ & $\wil{e}_{{367}}, \ \wil{e}_{{345}} ,\ \wil{e}_{{147}}$ 
\\[.1cm]
\hline
$1$ & $\wil{e}_{{356}},\  \wil{e}_{{247}},\ \wil{e}_{{167}},\ \wil{e}_{{145}},\ \wil{e}_{{123}}$ 
\end{tabular}
\end{center}

\begin{Theorem}\label{T:torus fixed points}
Among the eigenvectors of $t_{\lambda,\mu}$ in $\bigwedge^3 \II$, only the images of the following vectors in $\PP(\bigwedge^3 \II)$ 
lie in $X_{min}$: 
\begin{center}
\begin{tabular}{c|c}
${\lambda}^{2}{\mu}^{2}$ & $\wil{e}_{{357}}$ 
\\[.1cm]
\hline
${\frac {1}{{\lambda}^{2}{\mu}^{2}}}$ & $ \wil{e}_{{246}}$
\\[.1cm]
\hline
${\frac {1}{{\lambda}^{3}\mu}}$ & $\wil{e}_{{126}}$
\\[.1cm]
\hline
${\lambda}^{3}\mu$ & $\wil{e}_{{137}}$
\\[.1cm]
\hline
${\mu}^{2}$ & $\wil{e}_{{157}}$
\\[.1cm]
\hline
$\frac{1}{{\mu}^{2}}$ & $\wil{e}_{{146}}$ 
\\[.1cm]
\hline
$\frac{1}{\lambda^{4}}$ & $\wil{e}_{{256}}$
\\[.1cm]
\hline
${\lambda}^{4}$ & $\wil{e}_{{347}}$ 
\\[.1cm]
\hline
${\frac {\mu}{{\lambda}^{3}}}$ & $\wil{e}_{{125}}$ 
\\[.1cm]
\hline
${\frac {{\lambda}^{3}}{\mu}}$ & $\wil{e}_{{134}}$
\\[.1cm]
\hline
${\frac {{\mu}^{2}}{{\lambda}^{2}}}$ & $\wil{e}_{{257}}$ 
\\[.1cm]
\hline
${\frac {{\lambda}^{2}}{{\mu}^{2}}}$ & $\wil{e}_{{346}}$
\\[.1cm]
\hline
$1$ & $\wil{e}_{{167}},\ \wil{e}_{{145}},\ \wil{e}_{{123}}$
\end{tabular}
\end{center}
\end{Theorem}

\begin{proof}
It is easily checked that the points that are given in the hypothesis of the theorem are all torus fixed and all of them lie in $X_{min}$.
The only place we have to be careful is that $X_{min}$ may intersect eigenspaces of dimension $\geq 2$. 
Nevertheless, this potential problem does not occur; when we substitute a nontrivial linear combination of eigenvectors 
belonging to the same eigenvalue into equations (\ref{A:new 7a})--(\ref{A:new 7h}), we get a contradiction. 
\end{proof}

\section{Smoothness}

In this section, we will prove our main result. First, we have a remark on the dimensions.

\begin{Remark}\label{R:dimensions}
The dimension of $\mt{G}_2/ \mt{SO}_4$ is equal to 
$\dim \mt{G}_2 - \dim \mt{SO}_4=14-6=8$. 
We already pointed out in Remark~\ref{R:already} that 
$\mt{G}_2 / \mt{SO}_4$ is an affine subvariety of $X_{min}$, 
therefore, the dimension of $X_{min}$ is at least 8.
\end{Remark}

Next, we recall two standard facts.

\begin{enumerate}
\item  {\em Jacobian Criterion for Smoothness}: 
Let $I=(f_1,\dots, f_m)$ be an ideal from $\C[x_1,\dots, x_n]$ and let $x\in V(I)$ be a point from the vanishing locus of $I$ in $\C^n$.
Suppose $d= \dim V(I)$. If the rank of the Jacobian matrix $( \partial f_i/ \partial x_j )_{i=1,\dots, m,\ j=1,\dots, n}$ at $x$ is equal to $n-d$,
then $x$ is a smooth point of $V(I)$.


\item {\em Open charts on the Grassmannian}:
To see the complex manifold structure on $\mt{Gr}(d,\C^n)$, one looks at the intersections of $\mt{Gr}(d,\C^n)$ with the 
standard open charts in $\PP( \wedge^d \C^n)$:
$$
U_I := \mt{Gr}(d,\C^n) \cap \{ x\in \PP(\bigwedge^d \C^n):\ p_I(x) \neq 0\}.
$$
It is not difficult to show that the coordinate functions on $U_I$ are given by 
$p_J/p_I$, where $J=j_1\dots j_d$ is a sequence satisfying $|\{ j_1,\dots, j_d \} \cap \{ i_1,\dots, i_d \}|=d-1$. 
Indeed, it is not difficult to verify (by using Pl\"ucker relations) that any other rational function of the form $p_K/p_I$ is a 
polynomial in $p_J/p_I$'s. 

\end{enumerate}

\begin{Theorem}\label{T:smooth}
The algebraic set $X_{min}$ is a nonsingular projective variety of dimension 8. 
\end{Theorem}

\begin{proof}

Since $X_{min}$ is a closed set (defined as the intersection of certain hyperplanes with the Grassmann variety)
in a projective space, any irreducible component of $X_{min}$ is a projective variety. 
Moreover, since $X_{min}$ is stable under a torus action, each of these components is stable under the torus action as well.
By Borel Fixed Point Theorem~\cite[Theorem 10.4]{B}, we know that any irreducible component of $X_{min}$ 
contains at least one torus fixed point. In fact, there is a much stronger statement: Let $V$ be a vector space
and let $Y\subset \PP(V)$ be a projective $T$-variety, where $T$ is an algebraic torus.
Finally, let $Y^T$ denote the fixed point set of the torus action. In this case, $Y^T$ contains at least $\dim Y+1$
points. See~\cite[Lemma 2.4]{Carrell}. In Theorem~\ref{T:torus fixed points}, we showed that there are 
in total 15 torus fixed points in $X_{min}$. 
In the next few paragraphs we will show that each of these torus fixed points is smooth and 
its tangent space is 8 dimensional. Hence, each irreducible component of $X_{min}$ is 8 dimensional,
each component has at least 8+1=9 torus fixed points. 
A point in the intersection of two components is necessarily singular in $X_{min}$, hence, 
the Zariski tangent space at such a point would be at least $9$ dimensional. In other words, the irreducible 
components of $X_{min}$ do not intersect each other.
But this implies that there is only one irreducible 
component, otherwise, in $X_{min}$ there would at least be 18 torus fixed points. This finishes the proof of irreducibility. 

We proceed to show that the torus fixed points are smooth.
Note that the existence of a singular point in a $T$-variety implies the existence of a (possibly different) torus fixed singular point.
Thus, it suffices to analyze neighborhoods of fixed points by using affine charts that are described earlier.

We start with the fixed point $m=[\wil{e}_{123}]$, which lies on the open chart $\wil{U}_{123}$ as its origin. 
Here, ``tilde'' indicates that we are using transformed Pl\"ucker coordinates.
Recall that $X_{min}$ is cut-out on $\wil{U}_{123}$ by the vanishing of the seven linear forms (\ref{A:new 7a})--(\ref{A:new 7h}).
A straightforward calculation shows that the Jacobian of these polynomials with respect to variables
$\wil{q}_{{124}},\wil{q}_{{125}},\wil{q}_{{126}},\wil{q}_{{127}},\wil{q}_{{134}},
\wil{q}_{{135}},\wil{q}_{{136}},\wil{q}_{{137}},\wil{q}_{{234}},\wil{q}_{{235}},\wil{q}_{{236}},\wil{q}_{{237}}$ 
(in the written order) evaluated at the origin (which is $\wil{e}_{123}$) is equal to 
\setcounter{MaxMatrixCols}{12}
\begin{align*}
\mt{Jac}(\wil{f}_1,\dots, \wil{f}_7)|_{\wil{q}_{ijk}=0}=
\begin{pmatrix}
0&0&0&0&0&0&0&0&0&0&0&0 \\
0&0&0&0&0&0&0&0&0&0&0&0\\ 
0&0&0&0&0&0&0&0&0&0&0&0\\ 
0&0&0&1/2&0&0&-1/2&0&1/4&1/4&0&0 \\
0&0&0&-1/2&0&0&-1/2&0&1/4&-1/4&0&0\\
1/2&0&0&0&0&-1/2&0&0&0&0&-1/4&-1/4\\
1/2&0&0&0&0&1/2&0&0&0&0&-1/4&1/4
\end{pmatrix}
\end{align*}
which is obviously of rank 4. Hence, the dimension of the Zariski tangent 
space of $X_{min}$ at $m$ is $12-4=8$ dimensional. 
By Remark~\ref{R:dimensions}, we know that $\dim X_{min} \geq 8$,
hence we have the equality, $\dim X_{min} = 8$. In particular, 
$m$ is a smooth point of $X_{min}$.

We repeat this procedure for the other torus fixed points, which is tedious now. We verified this by using Maple. 
The outcome for each of the torus fixed points that are listed in Theorem~\ref{T:torus fixed points} turns out to be the same. 
In summary, all of the 15 torus fixed points on $X_{min}$ are nonsingular points, therefore, 
$X_{min}$ is a smooth projective variety of dimension 8.
\end{proof}

\begin{proof}[Proof of Theorem~\ref{T:main1:intro}]
The $\mt{G}_2$-orbit $\mt{G}_2/\mt{SO}_4 \subset X_{min}$ is irreducible and its dimension is equal 
to that of $X_{min}$. The proof follows from Theorem~\ref{T:smooth}. 
\end{proof}


\section{Tangent space at $[\wil{e}_{123}]$}\label{S:sample calculation}

In this section we perform a sample calculation of the weights of a generic one-parameter $\gamma: \C^* \rightarrow t_{\lambda,\mu}$ 
subgroup on the tangent space of $X_{min}$ at the $t_{\lambda,\mu}$-fixed point $m=[\wil{e}_{123}] \in X_{min}$. 
We use the term ``generic'' in the algebraic geometric sense, which is equivalent to the statement that 
the pairing between $\gamma$ and any character $\alpha : t_{\lambda,\mu} \rightarrow \C^*$
is nonzero. In other words, we choose a regular one-parameter subgroup $\gamma$ of $t_{\lambda,\mu}$ so that 
the fixed point set of $\gamma$ on $X_{min}$ is the same as that of $t_{\lambda,\mu}$. 
For example, $\gamma(s):= t_{s^{10},s}$, $s\in \C^*$ is regular.

Recall that the tangent space at $p=(p_1,\dots, p_n)$ of an affine variety $V\subseteq \C^n$ defined by the vanishing of the polynomials 
$f_1,\dots,f_r \in \C[x_1,\dots, x_n]$ is the intersection of the hyperplanes
$$
\sum_{i=1}^n \frac{\partial f_j}{\partial x_i} (p) (x_i-p_i) = 0,\ \text{  for } j =1,\dots, r.
$$
(Here we are abusing the notation. 
To be precise, $x_i$ should be replaced by the vector field $\partial/\partial x_i$.)
Equivalently, $T_p V$ is the kernel of the Jacobian matrix of $f_1,\dots,f_r$ (with respect to $x_i$'s) evaluated at the point $p\in V$. 
In our case, $p$ is $m=[\wil{e}_{123}]$ (the origin of the tangent space) and the Jacobian with respect to local coordinates 
$$
\wil{q}_{{124}},\wil{q}_{{125}},\wil{q}_{{126}},\wil{q}_{{127}},\wil{q}_{{134}},
\wil{q}_{{135}},\wil{q}_{{136}},\wil{q}_{{137}},\wil{q}_{{234}},\wil{q}_{{235}},\wil{q}_{{236}},\wil{q}_{{237}}
$$ 
is as given in the proof of Theorem~\ref{T:smooth}. It is straightforward to verify that 
\begin{align}\label{A:basis vectors for 123}
\left\{ -\frac{1}{2}x_{{135}}+x_{{237}},\frac{1}{2}x_{{124}}+x_{{236}},-\frac{1}{2}x_{{127}}+x_{{235}},
\frac{1}{2}x_{{136}}+x_{{234}},x_{{137}},x_{{134}},x_{{126}},x_{{125}} \right\}
\end{align}
is a basis for the kernel of the Jacobian matrix computed in the proof of Theorem~\ref{T:smooth}. 
Here, $x_{ijk}$ stands for the tangent vector $\frac{\partial}{\partial \wil{q}_{ijk}}$.

Recall from Section~\ref{S:two SLs} that $t_{\lambda,\mu}$ acts on $\II$ according to 
\begin{align*}
t_{\lambda,\mu} ( \wil{e}_1 ) &= \wil{e}_1 \\
t_{\lambda,\mu} ( \wil{e}_2 ) &= \frac{1}{\lambda^2}\ \wil{e}_2 \\
t_{\lambda,\mu} ( \wil{e}_3 ) &= \lambda^2\ \wil{e}_3 \\
t_{\lambda,\mu} ( \wil{e}_4 ) &= \frac{\lambda}{\mu}\ \wil{e}_4 \\
t_{\lambda,\mu} ( \wil{e}_5 ) &= \frac{\mu}{\lambda}\ \wil{e}_5 \\
t_{\lambda,\mu} ( \wil{e}_6 ) &= \frac{1}{\lambda \mu}\ \wil{e}_6 \\
t_{\lambda,\mu} ( \wil{e}_7 ) &= \lambda \mu\ \wil{e}_7. \\
\end{align*}
Let us denote by $w_{ijk}(\lambda,\mu)$ the weight (the eigenvalue) of the action $t_{\lambda,\mu} \cdot \wil{e}_{ijk}$.

The action of $t_{\lambda,\mu}$ on a Pl\"ucker coordinate $\wil{p}_{ijk}$ is given by
$$
t_{\lambda,\mu} \cdot \wil{p}_{ijk} (x) = \wil{p}_{ijk} ( t_{\lambda^{-1},\mu^{-1}}\cdot x)= w_{ijk}(\lambda^{-1},\mu^{-1}) \wil{p}_{ijk},
$$ 
and therefore, its action on a local Pl\"ucker coordinate function $\wil{q}_{rst}$ on $\wil{U}_{ijk}$ is given by 
$$
t_{\lambda,\mu} \cdot \wil{q}_{rst} (x) = \frac{ \wil{p}_{rst} ( t_{\lambda^{-1},\mu^{-1}}\cdot x) }{ \wil{p}_{ijk} ( t_{\lambda^{-1},\mu^{-1}}\cdot x) }
= \frac{ w_{rst} ( \lambda^{-1},\mu^{-1}) }{ w_{ijk}(\lambda^{-1},\mu^{-1}) } \wil{q}_{rst}.
$$
Consequently, if $v = \sum_{r,s,t} a_{rst} \frac{\partial}{\partial \wil{q}_{rst}}$ is a tangent vector at $\wil{e}_{ijk}\in \wil{U}_{ijk}$, then 
the action of the one-parameter subgroup $\gamma(\lambda) = t_{\lambda^{10},\lambda}$, $\lambda\in \C^*$ on $v$ is given by 
\begin{align}
\gamma \cdot v &= \sum_{r,s,t} \lim_{\lambda \to 1} \left(  \frac{ w_{rst}(\lambda^{10},\lambda)}{ w_{ijk}(\lambda^{10},\lambda)} \right) 
\frac{\partial}{\partial \wil{q}_{rst}}.
\end{align}
For example, the action of $\gamma$ on the basis vectors (\ref{A:basis vectors for 123}), which we denote by $v_1,\dots, v_8$ in the written order, 
is given by 
\begin{align*}
\gamma \cdot v_1 &= 11 v_1 \\
\gamma \cdot v_2 &= -11 v_2 \\
\gamma \cdot v_3 &= -9 v_3 \\
\gamma \cdot v_4 &= 9 v_4 \\
\gamma \cdot v_5 &= 31 v_5 \\
\gamma \cdot v_6 &= 29 v_6 \\
\gamma \cdot v_7 &= -31 v_7 \\
\gamma \cdot v_8 &= -29 v_8.
\end{align*}

\section{Bia\l ynicki-Birula Decomposition}

Let $X$ be a smooth projective variety over $\C$ on which an algebraic torus $T$ acts with finitely many fixed points.
Let $T'$ be a 1 dimensional subtorus with $X^{T'} = X^T$. For $p\in X^{T'}$, define the sets 
$$
C_p^+ = \{y \in X:\ {\lim_{s \to 0} s \cdot y = p,}\ s \in T' \}
$$
and 
$$
C_p^- = \{y \in X:\ {\lim_{s \to \infty} s \cdot y = p,}\ s \in T' \},
$$ 
called the plus and minus cells of $p$, respectively.

The following result is customarily called the Bia{\l}ynicki-Birula decomposition theorem in the literature. 


\begin{Theorem}[\cite{BB}]\label{T:BB}
If $X,T$ and $T'$ are as in the above paragraph, then  
\begin{enumerate}
\item both of the sets $C_p^+$ and $C_p^-$ are locally closed subvarieties in $X$, furthermore they are isomorphic to an affine space;
\item if $T_p X$ is the tangent space of $X$ at $p$, then $C_p^+$ (resp., $C_p^-$) is $T'$-equivariantly 
isomorphic to the subspace $T_p^+ X$ (resp., $T_p^- X$) of $T_p X$ spanned by the positive (resp., negative) 
weight spaces of the action of $T'$ on $T_p X$.
\end{enumerate}
\end{Theorem}

As a consequence of the $BB$-decomposition, there exists a filtration
$$
X^{T'}  = V_0 \subset V_1 \subset \cdots \subset V_n = X,\qquad n = \dim X,
$$
of closed subsets such that for each $i=1,\dots,n$, $V_i - V_{i-1}$ is the disjoint union of
the plus (resp.,  minus) cells in $X$ of (complex) dimension $i$. It follows that the odd-dimensional integral 
cohomology groups of $X$ vanish, the even-dimensional integral cohomology groups of $X$ are free, 
and the Poincar\'e polynomial $P_X(t) := \sum_{i=0}^{2n} \dim H^{i}(X; \C) t^i$ of $X$ is given by
$$
P_X(t) = \sum_{p \in X^{T'}} t^{2 \dim C_p^+} = \sum_{p \in X^{T'}} t^{2 \dim C_p^-}.
$$

Now, let $T'$ denote the 1 dimensional subtorus of $T=t_{\lambda,\mu}$ that is given by the image of 
the regular one-parameter subgroup $\gamma(\lambda) = t_{\lambda^{10},\lambda}$, $\lambda \in \C^*$. 
In the rest of this section, we are going to compute the weights of $T'$ on the tangent spaces at the torus fixed points.
We have already made a sample calculation of this sort in Section~\ref{S:sample calculation}.

\begin{enumerate}
\item $p=[\wil{e}_{{246}}]$. An eigenbasis for tangent space at $p$ is given by 
\[
\left\{
-x_{{234}}+x_{{467}}, 
-1/2x_{{124}}+x_{{456}},
x_{{346}},
x_{{245}}+x_{{267}},
x_{{256}},
1/2x_{{124}}+x_{{236}},
x_{{146}},
x_{{126}} 
\right\} 
\]
The weights (in the order of the eigenvectors) are $31,11,40,2,-18,11,20,-9$.

\item $p=[\wil{e}_{{157}}]$. An eigenbasis for tangent space at $p$ is given by
\[
\left\{
x_{{125}},
-1/2\,x_{{127}}+x_{{567}},
1/2\,x_{{135}}+x_{{457}},
x_{{357}},
x_{{257}},
x_{{167}},
x_{{145}},
x_{{137}}
\right\}
\]
The weights (in the order of the eigenvectors) are $-31,-11,9,20,-20,-2,-2,29$.

\item $p=[\wil{e}_{{256}}]$. An eigenbasis for tangent space at $p$ is given by
\[
\{
-x_{{235}}+x_{{567}},
x_{{236}}+x_{{456}},
1/2\,x_{{156}}+e_{{267}},
x_{{257}},
x_{{246}},
-1/2\,x_{{156}}+x_{{245}},
x_{{126}},
x_{{125}}
\}
\]
The weights (in the order of the eigenvectors) are $31,29,10,22,18,10,9,11$.

\item $p=[\wil{e}_{{126}}]$. An eigenbasis for tangent space at $p$ is given by
\[
\{
1/2\,x_{{156}}+x_{{267}},
x_{{256}},x_{{246}},
1/2\,x_{{124}}+x_{{236}},
x_{{167}},
x_{{146}},
x_{{125}},
x_{{123}}
\}
\]
The weights (in the order of the eigenvectors) are $11,-9,9,10,31,29,2,31$.

\item $p=[\wil{e}_{{167}}]$. An eigenbasis for tangent space at $p$ is given by
\[
\{
-1/2\,x_{{127}}+x_{{567}},
1/2\,x_{{136}}+x_{{467}},
-1/2\,x_{{147}}+x_{{367}},
1/2\,x_{{156}}+x_{{267}},
x_{{157}},
x_{{146}},
x_{{137}},
x_{{126}}
\}
\]
The weights (in the order of the eigenvectors) are $-9,9,10,-10,2,-2,31,-31$.

\item $p=[\wil{e}_{{145}}]$. An eigenbasis for tangent space at $p$ is given by
\[
\{
1/2\,x_{{135}}+x_{{457}},
-1/2\,x_{{124}}+x_{{456}},
1/2\,x_{{147}}+x_{{345}},
-1/2\,x_{{156}}+x_{{245}},
x_{{157}},
x_{{146}},
x_{{134}},
x_{{125}}
\}
\]
The weights (in the order of the eigenvectors) are $11,-11,10,-10,2,-2,29,-29$.

\item $p=[\wil{e}_{{123}}]$. An eigenbasis for tangent space at $p$ is given by
\[
\{
-1/2\,x_{{135}}+x_{{237}},
1/2\,x_{{124}}+x_{{236}},
-1/2\,x_{{127}}+x_{{235}},
1/2\,x_{{136}}+x_{{234}},
x_{{137}},
x_{{134}},
x_{{126}},
x_{{125}}
\}
\]
The weights (in the order of the eigenvectors) are $11,-11,-9,9,31,29,-31,-29$.

\item $p=[\wil{e}_{{137}}]$. An eigenbasis for tangent space at $p$ is given by
\[
\{
-1/2\,x_{{147}}+x_{{367}},
x_{{357}},x_{{347}},
-1/2\,x_{{135}}+x_{{237}},
x_{{167}},
x_{{157}},
x_{{134}},
x_{{123}}
\}
\]
The weights (in the order of the eigenvectors) are $-11,-9,9,-10,-31,-29,-2,-31$.

\item $p=[\wil{e}_{{125}}]$. An eigenbasis for tangent space at $p$ is given by
\[
\{
x_{{257}},
x_{{256}},
-1/2\,x_{{156}}+x_{{245}},
-1/2\,x_{{127}}+x_{{235}},
x_{{157}},
x_{{145}},
x_{{126}},
x_{{123}}
\}
\]
The weights (in the order of the eigenvectors) are $11,-11,9,10,31,29,-2,29$.

\item $p=[\wil{e}_{{257}}]$. An eigenbasis for tangent space at $p$ is given by
\[
\{
x_{{157}},x_{{125}},
-1/2\,x_{{127}}+x_{{567}},
x_{{237}}+x_{{457}},
x_{{357}},
x_{{245}}+x_{{267}},
x_{{256}},
-1/2\,x_{{127}}+x_{{235}}
\}
\]
The weights (in the order of the eigenvectors) are $20,-11,9,29,40,-2,-22,9$.


\item $p=[\wil{e}_{{357}}]$. An eigenbasis for tangent space at $p$ is given by
\[
\{
x_{{137}},
-x_{{235}}+x_{{567}},
1/2\,x_{{135}}+x_{{457}},
x_{{345}}+x_{{367}},
x_{{347}},
x_{{257}},
-1/2\,x_{{135}}+x_{{237}},
x_{{157}}
\}
\]
The weights (in the order of the eigenvectors) are $9,31,-31,-11,-2,18,-40,-11,-20$.


\item $p=[\wil{e}_{{146}}]$. An eigenbasis for tangent space at $p$ is given by
\[
\{
x_{{346}},
x_{{246}},
x_{{167}},
x_{{145}},
x_{{134}},
x_{{126}},
1/2\,x_{{136}}+x_{{467}},
-1/2\,x_{{124}}+x_{{456}}
\}
\]
The weights (in the order of the eigenvectors) are $20,-20,2,2,31,-29,11,-9$.


\item $p=[\wil{e}_{{347}}]$. An eigenbasis for tangent space at $p$ is given by
\[
\{
-x_{{234}}+x_{{467}},
x_{{237}}+x_{{457}},
-1/2\,x_{{147}}+x_{{367}},
x_{{357}},
x_{{346}},
1/2\,x_{{147}}+x_{{345}},
x_{{137}},
x_{{134}}
\}
\]
The weights (in the order of the eigenvectors) are $-31,-29,-10,-18,-22,-10,-9,-11$


\item $p=[\wil{e}_{{134}}]$. An eigenbasis for tangent space at $p$ is given by
\[
\{
x_{{146}},
x_{{145}},
x_{{137}},
x_{{123}},
x_{{347}},
x_{{346}},
1/2\,x_{{147}}+x_{{345}},
1/2\,x_{{136}}+x_{{234}}
\}
\]
The weights (in the order of the eigenvectors) are $-31,-29,2,-29,11,-11,-9,-10$.

\item $p=[\wil{e}_{{346}}]$. An eigenbasis for tangent space at $p$ is given by
\[
\{
1/2\,x_{{136}}+x_{{234}},
x_{{146}},
x_{{134}},
1/2\,x_{{136}}+x_{{467}},
x_{{236}}+x_{{456}},
x_{{345}}+x_{{367}},
x_{{347}},
x_{{246}}
\}
\]
The weights (in the order of the eigenvectors) are $-9,-20,11,-9,-29,2,22,-40$.


\end{enumerate}

\begin{Theorem}\label{T:Poincare polynomial}
The Poincar\'e polynomial of $X_{min}$ is 
$$
P_X(t^{1/2}) = 1+ t + 2t^2 + 2t^3 + 3t^4 + 2t^5 + 2 t^6+ t^7 + t^8.
$$
\end{Theorem}

\begin{proof}
The proof follows from the discussion at the beginning of this section and the computations made above.  
\end{proof}

\begin{Corollary}\label{C:ours}
The Picard number of $X_{min}$ is 1. 
\end{Corollary}
\begin{proof}
For a nonsingular projective variety $X$ over $\C$, the Picard number $\rho(X)$ of $X$ 
satisfies $1\leq \rho(X) \leq b_2$, where $b_2$ is the second Betti number of $X$. 
In the light of this fact, the proof follows from Theorem~\ref{T:Poincare polynomial}.
\end{proof}

\begin{Remark}\label{R:Ruzzi}
In~\cite[Theorem 2]{Ruzzi10}, Ruzzi showed that there exists unique 
smooth equivariant completion of $\mt{G}_2/\mt{SO}_4$ with Picard number 1. 
It follows from our Corollary~\ref{C:ours} that $X_{min}$ is the completion that 
Ruzzi found.
\end{Remark}

We finish our paper with a general remark. 

\begin{Remark}
The theory of equivariant embeddings of symmetric varieties is a very active and 
fascinating branch of algebraic geometry (see~\cite{Brion12,Timashev}). 
There are many $\mt{G}_2$-equivariant compactifications of $\mt{G}_2/\mt{SO}_4$.
For example, there is the well known ``wonderful compactification'' due to 
DeConcini and Procesi,~\cite{DP83}. The basic invariants of this compactification 
are determined in~\cite{DS85}. More general than the wonderful compactification are 
the ``special embeddings'' of symmetric varieties (see~\cite{CDM} and~\cite{Vust90}).
For of $\mt{G}_2/\mt{SO}_4$, there are many special embeddings whose 
posets of $\mt{G}_2$-orbits are isomorphic to that of $X_{min}$. 
However, these special embeddings (except the wonderful compactification) 
of $\mt{G}_2/\mt{SO}_4$ are not smooth.
\end{Remark}

\end{document}